\documentclass[a4paper,12pt]{article}

\usepackage[english]{babel}
\usepackage[T1]{fontenc}
\usepackage{amssymb}
\usepackage{amsmath}
\usepackage{amsthm}
\usepackage{mathrsfs}
\usepackage[all]{xy}

\usepackage{sectsty}

\sectionfont{\fontsize{14}{15}\selectfont}

\theoremstyle{thmit} 
\newtheorem{thm}{Theorem}[section]
\newtheorem{lem}[thm]{Lemma}
\newtheorem{cor}[thm]{Corollary}
\newtheorem{prop}[thm]{Proposition}
\newtheorem{question}{Question}
\newtheorem{deff}[thm]{Definition}
\newcommand{\Z}{\mathbb{Z}}

\newcommand{\A}{{\cal A}}

\newcommand{\Oe}{\widehat{OE}}
\newcommand{\B}{{\cal B}}

\newcommand{\G}{{\cal G}}

\newtheorem{rem}[thm]{Remark}

\usepackage{amssymb}
\usepackage{amsmath}
\usepackage{amsthm}
\usepackage{mathrsfs}
\usepackage[all]{xy}
\usepackage{graphicx,color}

\title{The profinite completion of accessible groups.}
\author{Vagner R. de Bessa, \ Anderson L. P. Porto and \ Pavel A. Zalesskii}

\begin{document}

\maketitle

\begin{abstract} We introduce a class $\A$ of finitely generated residually finite accessible groups with some natural restriction on one-ended vertex groups in their JSJ-decompositions.  We prove that the profinite completion of  groups in $\A$  almost detects its JSJ-decomposition and compute the genus of free products  of groups in $\A$. 
\end{abstract}

\section{Introduction}

There has been much recent study of whether residually finite groups, or classes of residually finite groups of combinatorial nature may be distinguished from each other by their sets of finite quotient groups.

In group theory the study in this direction started in 70-th of the last century when Baumslag [Bau74],  Stebe [Ste72] and others found examples of non-isomorphic finitely generated residually finite groups having the same set of finite quotients. The general question  addressed in this study can be formulated as follows: 
		
		\begin{question}\label{question}
			To what extent a finitely generated residually finite group $\Gamma$ is  determined by its finite quotients?
		\end{question} 
		
		The study leaded to the notion of genus $\mathfrak{g}(G)$ of a finitely generated residually finite group $G$, the set of isomorphism classes of finitely generated residually finite groups having the same set of finite quotients as $G$. Equivalently, $\mathfrak{g}(G)$ is the set of isomorphism classes of finitely generated residually finite groups having the  profinite completion isomorphic to the profinite completion $\widehat G$ of $G$. In fact, the term genus was  borrowed from integral representation theory, where for a finite group $K$  the genus of a $\Z K$-lattice $M$ is defined as the set of isomorphism classes of $\Z K$-lattices $N$ such that the $\widehat \Z K$-modules $\widehat M$ and $\widehat N$ are isomorphic.    
		
		The study mostly was concentrated to establish whether the cardinality $g(G)$ of  the  genus $\mathfrak{g}(G)$ is finite or  1 (see  \cite{GPS, GS, GZ, BZ, BCR16, BMRS18, BMRS20, Wil17} for example) (we use the same term {\it genus} for $g(G)$ from now on). 
However, the following 		
		 principle question of Remeslennikov is still open.
		
		\begin{question} (V.N. Remeslennikov)
		Is the genus $g(F)$ of a free group $F$ of finite rank equal to 1?
		
		\end{question}

Note that this question is easily reduced to the question whether the profinite completion of finitely generated residually finite one-ended group can be free profinite. In particular, it is not clear whether the profinite completion of an one-ended group does not split as a profinite amalgamated free product or an HNN-extension over a finite group. This naturally gives restriction to the family in which we make our considerations.

We define the family $\A$ to consist of all finitely generated residually finite accessible groups $G$ such that if $G$ splits as an amalgamated free product $G=G_1*_HG_2$ or an HNN-extension $G=HNN(G_1,H,t)$ over a finite group $H$ then $G_i\in \A$ and if $G\in \A$ is one-ended, then $\widehat G$ can not act on a profinite tree with finite edge stabilizers without a global fixed point (see Definition \ref{Profinite tree} for the definition of a profinite tree). Note that $\A$ is closed for free products with finite amalgamation and HNN-extensions with finite associated subgroups (Proposition \ref{closed for constructions}). It contains all finitely generated residually finite soluble groups (and in general groups satifying a law), Fuchsian groups, 3-manifold groups as well all arithmetic groups of rank $\geq 2$, Fab groups including Grigorchuk and Gupta-Sidki type groups and many more.  
 
In this paper we study the profinite genus of groups within this family $\A$. 

We first show that the profinite completion of a group in $\A$ almost determines its JSJ-decomposition,  i.e.  a decomposition as the fundamental group of a finite graph of groups with finite edge groups and finite or one-ended vertex groups.

\begin{thm}\label{JSJ-decomposition} Let $G,B\in \A$ such that $\widehat G\cong \widehat B$ and $G=\pi_1(\G, \Gamma)$, $B=\pi_1(\B, \Delta)$ be their JSJ-decompositions. Then there  are bijections $\epsilon:E(\Gamma)\longrightarrow E(\Delta)$ of  sets of edges and  $\varphi: V(\Gamma)\longrightarrow V(\Delta)$ of sets of vertices such that $\G(e)\cong \B(\epsilon(e))$ and $\widehat{\G(v)}\cong \widehat{\B(\varphi(v))}$ for all $e\in E(\Gamma), v\in V(\Gamma)$.

\end{thm} 

Using this we deduce that the profinite completion  of a group in $\A$ determines the Grushko decomposition, i.e.  a decomposition into a free product of indecomposable factors. This generalizes \cite[Theorem A]{WZ}  and \cite[Proposition
6.2.4]{wilkes}, where this was proved for 3-manifold groups.

\begin{cor}\label{free product}  Let $G_1, \ldots, G_n $ be groups in $\A$ indecomposable into a free product  and $G=*_{i=1}^n G_i$ be their free product.  Let $B\in \A$ with $\widehat G=\widehat B$. Then $B=*_{i=1}^n B_i$, with $\widehat B_i\cong \widehat G_i$ for all $i=1,\ldots, n$.  

\end{cor}

 This in turn allows to prove that the genus $g(G,\A)$  is multiplicative with respect to free products.

  \begin{thm}\label{multiplicative} Let $G_1, \ldots, G_n$ be groups in  $\A$  and $G=*_{i=1}^n G_i$ be their free product. Then $$g(G,\A)=g(G_1,\A)g(G_2,\A) \cdots g(G_n,\A).$$

\end{thm}

We also show that the decomposition of groups from class $\A$ into a free product can be characterized in terms of their profinite completions. In particular, this holds for finitely generated virtually free groups (see Corollary \ref{virtually free}).

\begin{thm}\label{free product decomposition} Let $G$ be a group in $\A$.  Then $\widehat G$ splits as a free profinite product if an only if $G$ splits as a free product.

\end{thm}

The structure of the  paper is as follows. Section 2 contains elements of the profinite version of the Bass-Serre theory used in the paper (see \cite{R} for more details). Section 3 focuses on  amalgamated free products of groups with less than 2 ends and their profinite completion. Theorem \ref{JSJ-decomposition} is proved in Section 4.

\bigskip
Our basic reference for notations and results about profinite groups is [RZ]. We refer the reader to Lyndon-Schupp [LS], Magnus-Karrass-Solitar [MKS], Serre [S] or Dicks-Dunwoody [DI] for an account of basic facts on amalgamated free products  and to [RZ] for the profinite versions of these constructions. Our methods based on the profinite version of the Bass-Serre theory of groups acting on trees that can be found in \cite{R}. All homomorphisms of profinite groups are assumed to be continuous in this paper. We will use the standard abbreviation $x^{g}$ for $g^{-1}xg$ when $g,x$ are elements of a group $G;$ the inner automorphism of $G$ corresponding to this conjugation will be denoted by $\tau_g$. For a subgroup $H$ of $G$ the notation $\langle\langle H\rangle\rangle$ will stand for the normal closure of $H$ in $G$. The composition of the two applications $f$ and $g$ is often defined simply as $f\circ g = fg.$ All amalgamated free products $G=G_1*_HG_2$ (resp. profinite amalgamated free products $G=G_1\amalg_HG_2$) will be assumed non-fictitious in the paper, i.e. $G_1\neq H\neq G_2$.

\section{\large Preliminary Results}

In this section we recall the necessary notions of the Bass-Serre theory for abstract and profinite graphs.
	
	\begin{deff}[Profinite graph]
		A (profinite) graph is a (profinite space) set $\Gamma$ with a distinguished (closed) nonempty subset $V(\Gamma)$ called the vertex set, $E(\Gamma)=\Gamma-V(\Gamma)$ the edge set and two (continuous) maps $d_0,d_1:\Gamma \rightarrow V(\Gamma)$ whose restrictions to $V(\Gamma)$ are the identity map $id_{V(\Gamma)}$. We refer to $d_0$ and $d_1$ as the incidence maps of the (profinite) graph $\Gamma$.  
	\end{deff}
	
	A morphism $\alpha:\Gamma \longrightarrow \Delta$ of profinite graphs is a continuous map with $\alpha d_i=d_i \alpha$ for $i=0,1$. By \cite[Proposition 2.1.4]{R} every profinite graph $\Gamma$ is an inverse limit of finite quotient graphs of $\Gamma$.
	
	\begin{deff}\label{Profinite tree} Let $\Gamma$ be a profinite graph. Define $E^{*}(\Gamma)=\Gamma/V(\Gamma)$ to be the quotient space of $\Gamma$ (viewed as a profinite space) modulo the subspace of vertices $V(\Gamma)$. Consider the free profinite $\widehat\Z$-modules $[[\widehat\Z(E^{*}(\Gamma),*)]]$ and $[[\widehat\Z V(\Gamma)]]$ on the pointed profinite space $(E^{*}(\Gamma),*)$ and on the profinite space $V(\Gamma)$, respectively. Denote by $C(\Gamma,\widehat\Z)$ the chain complex

	$$	\xymatrix{ 0 \ar[r] & [[\widehat\Z(E^{*}(\Gamma),*)]] \ar[r]^d &[[\widehat\Z V(\Gamma)]] \ar[r]^{\varepsilon} & \widehat\Z \ar[r] & 0}$$
		of free profinite $\widehat\Z$-modules and continuous $\widehat\Z$-homomorphisms $d$ and $\varepsilon$ determined by $\varepsilon(v)=1$, for every $v \in V(\Gamma)$, $d(\overline{e})=d_1(e)-d_0(e)$, where $\overline{e}$ is the image of an edge $e \in E(\Gamma)$ in the quotient space $E^{*}(\Gamma)$, and $d(*)=0$. 
		One says that $\Gamma$ is a profinite tree if the sequence $C(\Gamma,\widehat\Z)$ is exact. 
	\end{deff}	
	
	  If $v$ and $w$ are elements of a tree (respectively profinite tree) $T$, one denotes by $[v,w]$ the smallest subtree (respectively profinite subtree) of $T$ containing $v$ and $w$.

\begin{deff}
Let $\Gamma$ be a connected finite graph. A    graph of profinite groups $({\cal G},\Gamma)$ over
$\Gamma$ consists of  specifying a profinite group ${\cal G}(m)$ for each $m\in \Gamma$, and continuous monomorphisms
$\partial_i: {\cal G}(e)\longrightarrow {\cal G}(d_i(e))$ for each edge
$e\in E(\Gamma)$, $i=0,1$. We say that it is reduced if $\G(e)\neq \G(v_i)$ for all edges of $\Gamma$ which are not loops.
\end{deff}

In \cite[paragraph (3.3)]{ZM1},  the fundamental group
 $G=\Pi_1(\G,\Gamma)$ is  defined explicitly in terms of generators and relations
 associated to a chosen maximal subtree $D$.  Namely the profinite presentation  is the same as in the abstract case and is   as follows:
 
 \begin{eqnarray} \label{presentation} 
  G=\Big\langle
 \G(v), t_e, v\in V(\Gamma), e\in E(\Gamma) \mid &  \nonumber\\ t_e=1, \ {\rm for}\  e\in D, \partial_0(g)=t_e\partial_1(g)t_e^{-1}, \ {\rm for}\ g\in \G(e)\Big\rangle & 
\end{eqnarray}
I.e., if one takes the abstract fundamental group $G_0=\pi_1(\G,\Gamma)$,
then $$\Pi_1(\G,\Gamma)=\varprojlim_N G_0/N,$$ where $N$ ranges over
all normal subgroups of $G_0$  with $N\cap
\G(v)$ open in $\G(v)$ for all $v\in V(\Gamma)$. Note that this last
condition is automatic if $\G(v)$ is finitely generated (as a
profinite group, see \cite{NS-07}). 
   It is also proved in \cite{ZM1}
that the definition given above is independent of the choice of
the maximal subtree $D$.

If all vertex groups are trivial we get the definition of the profinite fundamental group $\pi_1(\Gamma)$ that is the profinite completion of the usual fundamental group.

Note that in contrast with the classical case, the vertex groups of $(\G,\Gamma)$  do not always  embed in $\Pi_1(\G,\Gamma)$. However,   it is always the case if the edge groups are finite (see \cite[Proposition 6.5.1]{R}) that will be assumed for the rest of the paper. In particuular,  free  profinite products with amalgamation and HNN-extensions will be proper in the sense of \cite[Chapter 9]{RZ}.

\medskip
Associated with the profinite graph of profinite groups $({\cal G}, \Gamma)$ there is
a corresponding  {\em  standard profinite tree}  (or universal covering graph)
  $$S=S(G)=\bigcup_{m\in \Gamma}
G/\G(m)$$ (cf. \cite[Theorem 3.8]{ZM1}).  The vertices of
$S$ are those cosets of the form
$g\G(v)$, with $v\in V(\Gamma)$
and $g\in G$; its edges are the cosets of the form $g\G(e)$, with $e\in
E(\Gamma)$; and the incidence maps of $S$ are given by the formulas:

$$d_0 (g\G(e))= g\G(d_0(e)); \quad  d_1(g\G(e))=gt_e\G(d_1(e)) \ \ 
(e\in E(\Gamma), t_e=1\hbox{ if }e\in D).  $$

 There is a natural  continuous action of
 $G$ on $S$, and clearly $ G\backslash S= \Gamma$. 

\medskip 
 If $G=G_1*_HG_2$ is an amalgamated free product then $G=\pi_1(\G, \Gamma)$ with $\Gamma$ having two vertices and one edge, $G_1,G_2$ being vertex groups and $H$ being an edge group. In this case $S(G)=G/G_1\cup G/G_2\cup G/H$ and  $d_0 (gH)= gG_1; \quad  d_1(gH)=gG_2.$

\bigskip
\bigskip
We shall need slightly more general version of \cite[Lemma 4.4]{GZ}. 

\begin{lem}\label{generation} Let $G=G_1*_HG_2$ be an  amalgamated free product and let $g\in G$. Then $G=\langle G_1, G_i^g\rangle$ if and only if $i=2$ and $g\in G_2G_1$.

\end{lem}  

 The proof of the lemma is the same as the proof of  \cite[Lemma 4.4]{GZ}  and  will be ommited.

\begin{prop}\label{o421} 

\begin{enumerate}  

\item[(i)] Let $R = R_{1}\star_{L}\,R_2$. Then $R_i\cap R_j^r\leq   L^{b}$ for some $b \in R_i$, whenever $i\neq j$  or $r\not\in R_i$, $(i,j\in \{1,2\}).$ Moreover $N_{R}(K)= R_i $ for any normal subgroup $L\neq K\triangleleft R_i$ of $R_i$, $(i=1,2).$

\item[(ii)] Let $R = R_{1}\amalg_{L}\,R_2$ be a profinite amalgamated free product. Then $R_i\cap R_j^r\leq   L^{b}$ for some $b \in R_i$, whenever $i\neq j$  or $r\not\in R_i$, $(i,j\in \{1,2\}).$ Moreover $N_{R}(K)= R_i$ for any closed normal subgroup $L\neq  K\triangleleft R_i$  of $R_i$, $(i=1,2).$

\end{enumerate}

\end{prop}

\begin{proof} (ii) is Corollary 7.1.5 (b) in [R] or Corollary 3.13 in [ZM1]. The proof of (i) is the same using the classical Bass-Serre theory instead of the profinite one. The last part of the statement follows from the first taking $r\in N_{R}(K)$.

\end{proof}

\section{OE-groups}

A group $G$ with less than 2 ends, i.e. a finite or one-ended group, will be called an $OE$-group in the paper.   It follows from the famous Stallings theorem that $G$ is an $OE$-group  if and only if whenever it acts on a tree with finite edge stabilizers
it  has a global fixed point. A profinite group will be called $OE$-group if it has the same property:   whenever it acts   on a  profinite tree  with finite edge stabilizers, then it fixes a vertex. 

 A finitely generated residually finite group $G$  will be called $\widehat{OE}$-group if $\widehat G$ is $OE$-group. Note that an $\widehat{OE}$-group is automatically $OE$-group, since if a residually finite group $G$ splits as an amalgamated free product or HNN-extension over a finite group then so does $\widehat G$.  

The next proposition gives a sufficient condition for $G$ to be an $\Oe$-group.

\begin{prop}\label{virtually cyclic} Let $G$ be a finitely generated residually finite group such that $\widehat G$ does not have a non-abelian free pro-$p$ subgroup for any prime $p$.  If  $G$ is not virtually infinite cyclic, then $G$ is a $\Oe$-group.

\end{prop}

\begin{proof} Suppose $G$ is not an $\Oe$-group. Then it acts on a profinite tree $T$ with finite edge stabilizers and does not fix a vertex.  By  \cite[Theorem 4.2.11]{R} or \cite[Theorem 3.1
]{Z} if $\widehat G$ does  not fix a vertex,  then there exists a closed normal subgroup $K$ of $\widehat G$ fixing some edge such that the quotient group $\widehat G/K$   is either a projective metaprocyclic group $\widehat \Z_\pi\rtimes \widehat \Z_\rho$ ($\pi,\rho$ are set of primes with $\pi\cap \rho=\emptyset$) or an infinite soluble  profinite Frobenius group $\widehat\Z_\pi \rtimes C$  ($C$ finite cyclic) or an infinite profinite generalized dihedral group $\widehat\Z_\pi\rtimes C_2$. 

Therefore there exists a finite index subgroup $R$ of $G$ such that     $\widehat R$ is torsion free.  If $\widehat R$ is not procyclic, then $\widehat R\cong \widehat \Z_\pi\rtimes \widehat \Z_\rho$, $\pi\cap \rho=\emptyset$ with non-trivial action. It follows that $\widehat R/[\widehat R,\widehat R]=\widehat{R/[R,R]}$ can not contain $\widehat \Z$. But     $R/[R,R]$  is finitely generated infinite abelian and so contains $\Z$ implying that $\widehat{R/[R,R]}$ must contain $\widehat \Z$, a contradiction with the previous sentence. Thus $\widehat R$ is procyclic.
It follows that $R$ is finitely generated abelian and so is infinite cyclic.

\end{proof}

 The class of finitely generated residually finite $\widehat{OE}$- groups is quite large.  For example Fuchsian groups, 3-manifold groups and    all  arithmetic groups of rank $\geq 2$ are $\Oe$-groups.

If $G$   satisfies an identity, then $\widehat G$ satisfies the same identity and so does not have non-abelian free  pro-$p$ subgroups, so it is an $\Oe$-group by Proposition \ref{virtually cyclic} unless it is virtually infinite cyclic.

\bigskip

\begin{prop} \label{2a} Let $G = G_{1}\star_{H}\,G_2$ and $B = B_{1}\star_{K}\,B_2$ be   amalgamated free products of groups  with finite amalgamation.  Suppose  $G_i, B_i ( i=1,2)$ are  finitely generated residually finite $OE$-groups.  If $G\cong B$ then there exist an isomorphism $\psi: G \longrightarrow B$ such that $\psi(H)= K$, $\psi(G_1)= B_1$, $\psi(G_2)= B_2$  (up to possibly interchanging $B_1$ and $B_2$ in $B$).

\end{prop}

\begin{proof} Let $\psi: G\rightarrow  B$ be an isomorphism. Since $\psi(G_1)$ is an OE-group, it does not split as an amalgamated free product or HNN-extension over  a finite subgroup and so    is conjugate into $B_1$ or $B_2$ say $\psi(G_1)\leqslant {B_{1}}^{b}$ for some $b\in B$. Recall that $\tau_b$ denote the inner automorphism that  corresponds to conjugation by $b$ . Thus replacing $\psi$ by $\tau_{b^{-1}} \circ \psi$,  we may assume that  $\psi(G_1) \leqslant  B_1$. Symmetrically  $\psi^{-1}(B_1)$ is in $G_1^g$ for some $g\in G$. But $\psi^{-1}\psi(G_1)=G_1$  and so $\psi( G_1)= B_1$.   

Next we  show that in addition we may assume that $\psi(G_2)= B_2.$ Similarly as in the preceding paragraph  $\psi(G_2)$ is conjugate to $B_1$ or $B_2$.  Then by Lemma \ref{generation} there exists $b=b_2b_1$ such that $\psi(G_2)=B_2^{b}=B_2^{b_1}$ with $b_2\in B_2, b_1\in B_1$.  Thus replacing $\psi$ with its composite with conjugation by $b_1$ we have   $\psi(G_2)=B_2$.

Now we have $\psi(H) = B_{1} \cap B_{2} \leqslant K$.
 Similarly we have $\psi^{-1}(K)\leqslant H$. Thus   $\psi(H)=K$ and the proof is complete.

\end{proof}

Now we shall prove the profinite version of Proposition \ref{2a}. 

\begin{prop} \label{compart1} Let $G = G_{1}\amalg_{H}\,G_2$ and $B = B_{1}\amalg_{K}\,B_2$ be   profinite amalgamated free products of  profinite groups  with finite amalgamation. Suppose  $G_i,B_i\  (i=1,2)$ are OE-groups.   If $G\cong  B $ then there exist an isomorphism $\psi:G\rightarrow  B$ such that $\psi(H) = K$, $\psi(G_1) =B_{1}$ and $\psi(G_2) ={B_{2}}^{b}$ for some $b\in B$ (up to possibly interchanging $B_1$ and $B_2$ in $B$).

\end{prop}

\begin{proof} Let $\psi: G\rightarrow  B$ be an isomorphism. Then  $\psi(G_1)$ is conjugate into $B_1$ or $B_2$, say $\psi(G_1)\leqslant {B_{1}}^{w}$ for some $w\in B$ (cf. \cite[Example 6.3.1]{R}). Thus replacing $\psi$ by $\tau_{w^{-1}} \circ \psi$, if necessary, we may assume that $\psi(G_1) \leqslant B_1$. Then symmetrically  $\psi^{-1}(B_1)$ is in $G_1^g$ for some $g\in G$. 
But $\psi^{-1}\psi(G_1)=G_1$  and so $\psi(G_1)=B_1$. 

 Similarly  $\psi(G_2)$ is conjugate into $B_1$ or $B_2$. But $ G_1$ and $ G_2$ are not conjugate, since otherwise by Proposition \ref{o421} $G_1$ is conjugate to $H$ and  $G$ is fictitious. Thus  $\psi(G_2)\leqslant {B_{2}}^{b}$ for some $b\in B$ (cf. \cite[Example 6.3.1]{R}).  Then symmetrically  $\psi^{-1}(B_2^b)$ is in $G_2^g$ for some $g\in G$. 
But $\psi^{-1}\psi(G_2)=G_2$  and so $\psi(G_2)=B_2^b$.

By Proposition \ref{o421} we have: $$\psi(H) = B_{1} \cap  B_{2}^{b} \leqslant K^{b_1} \,(b_1\in B_1).$$ Similarly we have $\psi^{-1}(K)\leqslant H^{g_1}$ for some $g_1 \in G_1.$ Since $H$ and $K$ are finite we have $\psi(H)=K^{b_1},\, b_1\in B_1.$ Thus replacing $\psi$ by $\tau_{{b_1}^{-1}} \circ \psi$,  we get   $\psi(H)=K$ as required.

\end{proof}

\begin{rem} \label{finite-by-cyclic}
\begin{enumerate}

\item[(i)]
Proposition \ref{2a}  also holds if we assume that  $G_i, B_i$ ($i=1,2)$ are semidirect products   $M_i\rtimes \Z$, $N_i\rtimes \Z$ with  $M_i,N_i$ finite such that $M_i\neq H$, $N_i\neq K$. 

\item[(ii)]
Similarly, Proposition \ref{compart1}  also holds if we assume that $G_i, B_i$ ($i=1,2)$ are semidirect products   $M_i\rtimes \widehat\Z$, $N_i\rtimes \widehat\Z$  with $M_i,N_i$ finite such that $M_i\neq H$, $N_i\neq K$.

\end{enumerate}

\smallskip
The proofs are the same using Proposition \ref{o421}. 

\end{rem}

\section{Detecting JSJ-decomposition}

\begin{deff}\label{accessible}
A group $G$ is called accessible if there exists a natural number $n=n(G)$ such that for any splitting of $G$ as the fundamental group $\pi_1(\G, \Gamma)$ of a reduced graph of groups $(\G, \Gamma)$ with finite edge groups,  the size of $\Gamma$ is bounded by $n(G)$.\end{deff}

 Thus if $G$ is  accessible, then there exists $(\G, \Gamma)$ of maximal size where vertex groups are either finite or one-ended and therefore can not split further. A representation of $G$ as the fundamental group $\pi_1(\G,\Gamma)$ of such graph of groups will be called a JSJ-decomposition of $G$.   Note that  every  finitely presented group is accessible (see [D]) and so admits a JSJ-decomposition.

Recall that the class $\A$  consists of all finitely generated residually finite accessible groups such that if $G$ splits as an amalgamated free product $G=G_1*_HG_2$ or an HNN-extension $G=HNN(G_1,H,t)$ over a finite group $H$ then $G_i\in \A$ and whenever $G\in \A$ is one-ended, then $\widehat G$ can not act on a profinite tree with finite edge stabilizers without a global fixed point (see Definition \ref{Profinite tree} for the definition of a profinite tree). We first show that $\A$ is closed for free products with finite amalgamation and HNN-extensions with finite associated subgroups.

\begin{prop}\label{closed for constructions} Let $H$ be a finite subgroup of groups $G_1,G_2\in\A$. Then $G_1*_HG_2$ and $HNN(G_1,H,t)$ are  in $\A$.\end{prop}

\begin{proof}  Let $G_i=\pi_1(\G_i,\Gamma_i)$ be a JSJ-decomposition of $G_i$, $i=1,2$. Since $H$ is finite, by \cite[Theorem 3.10]{ZM1} or \cite[Theorem 7.1.2]{R} $H^{g_i}\leq \G_i(v_i)$, for some $g_i\in G_i$, $v_i\in V(\Gamma_i)$, $i=1,2$. Conjugating all vertex and edge groups by $g_i$ we have the graph of groups $(g_i\G_i g_i^{-1},\Gamma_i)$  with  $G_i=\pi_1(g_i\G_i g_i^{-1},\Gamma_i)$ and so  w.l.o.g. we may assume that $H\leq \G_i(v_i)$. If $G=G_1*_HG_2$ then connecting $v_1$, $v_2$ by an edge $e$ and setting $\G(e)=H$ we obtain a JSJ-decomposition of $G=G_1*_HG_2$ with vertex groups being the union of vertex groups of $(\G_1,\Gamma_1)$ and $(\G_2,\Gamma_2)$. This shows that $G_1*_HG_2\in \A$. If $G=HNN(G_1,H,t)$ then $H^t\leq \G(w)^g$ for some $g\in G_1$, $w\in V(\Gamma_1)$ and so $G=HNN(G_1,H,tg)$. Then conneting $v_1$ with $w$ by an edge $e$ and setting  $\G(e)=H$ we obtain a JSJ-decomposition of $HNN(G_1,H,tg)=HNN(G_1,H,t)$ with vertex groups being the  vertex groups of $(\G_1,\Gamma_1)$ that shows that $HNN(G_1, H, t)\in\A$. The proof is finished. 

\end{proof}

\begin{lem}\label{relative main} Let  $G=\pi_1(\G, \Gamma)$ be the fundamental group of a finite graph of finitely generated residually finite groups with finite edge groups.  Suppose $\widehat G$ acts on a profinite tree $T$ such that for each $v\in V(\Gamma)$ the profinite completion $\widehat{\G(v)}$ fixes a vertex of $T$ and  for any edge $f$ of $T$ one has $|\widehat G_f|< |\G(e)|$ for each edge group of $(\G,\Gamma)$. Then $\widehat G$ stabilizes a vertex of $T$. 

\end{lem}

\begin{proof} Note first that since the edge groups of $G$ are finite, there exists a normal sugroup $U$ of $G$ of finite index that intersect them trivially and so $U$ is a free product of its intersection with conjugates of vertex groups; so by \cite[Corollary 3.1.6]{RZ} the profinite topology of $G$ induces the full profinite topology on the vertex groups. This means that the profinite completion of vertex groups coincides with their closure in $\widehat G$. By  \cite[Theorem 2.10]{ZM1} or \cite[Theorem 4.1.8]{R} $\G(e)$ fixes a vertex $w$ in $T$.  By hypothesis the vertex groups $\widehat{\G(d_0(e))}$ and $\widehat{\G(d_1(e))}$ fix some vertices $v_0,v_1$ of $T$ respectively and therefore  by \cite[Corollary 4.1.6]{R} $\G(e)$ fixes the geodesics $[v_0,w]$ and $[v_1, w]$.   Then our hypothesis on the edge groups of $(\G,\Gamma)$ imply that $v_0=w=v_1$.  Since $\Gamma$ is connected,  we deduce that all vertex groups of $(\G, \Gamma)$ fix $w$.  Let $D$ be a maximal subtree of $\Gamma$.  Then we can view $G$ as HNN-extension $G=HNN(H, \G(e),  t_e)$,  $e\in \Gamma\setminus D$, where $H=\pi_1(\G,D)$ is the fundamental group of the tree of groups obtained by  restriction of $(\G,\Gamma)$ to $D$.  Since $H$ is generated by vertex groups, $H$ fixes $w$.  Then $\widehat H^{t_e}$ fixes $t_e^{-1}w$ and $\widehat H\cap \widehat H^{t_e}=\G(e)$  fixes $[t_e^{-1}w,w]$ by  \cite[Corollary 4.1.6]{R},   so by hypothesis $t_ew=w$ for each $e\in \Gamma\setminus D$.  Thus the result follows from the presentation (1). 

\end{proof}

{\it Proof of Theorem \ref{free product decomposition}}.  Suppose $\widehat G=A\amalg B$ splits as a free profinite product.  Let $S(\widehat G)$ be a standart profinite tree associated with this free profinite product.  By Lemma \ref{relative main} if all edge groups $\G(e)$ are non-trivial then $\widehat G$ fixes   a vertex,  a contradiction.  Therefore at least one edge group $\G(e)$ is trivial.  As $G=G_1*_{\G(e)}G_2$ splits  as a free amalgamated product  or  as $G=HNN(G_1,\G(e),t)$ over $\G(e)$ we deduce that $G=G_1* G_2$ or $G_1* \langle t\rangle$.

\bigskip

The next corollary answers a question of Andrei Jaikin asked in private communication.

\begin{cor}\label{virtually free} Let $G$ be a finitely generated virtually free group.  Then $G$ splits as a free product if and only if $\widehat G$ splits as a free profinite product. 

\end{cor}

\begin{proof} By a result of Karrass, Pietrovski and Solitar \cite{KPS-73} $G$ is the fundamental group of a finite graph of finite groups and therefore is in $\A$. Hence the result follows from Theorem \ref{free product decomposition}.

\end{proof}

\begin{prop} \label{maximal edge groups} Let $G=\pi_1(\G, \Gamma)$, $B=\pi_1(\B, \Delta)$ be the fundamental groups of reduced finite graphs of finitely generated residually finite groups with finite edge groups and  $\nu:\widehat B\longrightarrow \widehat G$ be an isomorphism.  Suppose that for any $v\in V(\Delta)$ the group  $\nu(\widehat{\B(v)})$ is conjugate into $\widehat{\G(w)}$ for some $w\in V(\Gamma)$ and for any $w\in V(\Gamma)$ the group $\nu^{-1}(\widehat{\G(w)})$ is conjugate into $\widehat{\B(v)}$ for some $v\in V(\Delta)$.  Then there  are bijections $\epsilon:E(\Gamma)\longrightarrow E(\Delta)$, $\varphi: V(\Gamma)\longrightarrow V(\Delta)$ such that   $\widehat{\G(w)}\cong \widehat{\B(\varphi(w))}$ for all $w\in V(\Gamma)$  and for every  edge group $\G(e)$ one has   $\G(e)\cong \B(\epsilon(e))$.

\end{prop}

\begin{proof}   Since $\nu(\widehat{\B(v)})$ is conjugate into $\widehat{\G(w)}$ for some $w\in V(\Gamma)$ and
$\nu^{-1}(\widehat{\G(w)})$ is conjugate   into $\widehat{\B(u)}$ for some $u\in V(\Delta)$,  we have $u= v$, because otherwise by \cite[Theorem 7.1.4]{R} $\widehat{\B(v)}$ is in some edge group of $B$ contradicting the hypothesis that $(\B, \Delta)$ is reduced. 
 But  $\nu^{-1}\nu\widehat{\B(v)}=\widehat{\B(v)}$ so  $\nu (\widehat{\B(v)})= \widehat{\G(w)}^g$ for some $g \in \G$. Therefore $(\tau_{g^{-1}} \nu)_{|\widehat{\B(v)}}$ is an isomorphism
to $\widehat{\G(w)}$.  So we can define $\varphi:V(\Gamma)\longrightarrow V(\Delta)$ by putting $\varphi(w)$ to be the unique vertex $v$ in $\Delta$ such that $\widehat{\G(w)}$ is conjugate to $\widehat{\B(v)}$.

The edge groups are finite subgroups that are intersections of two incident vertex groups. Therefore, every edge group  $\G(e)$ is isomorphic to a subgroup of some edge group  $\B(e')$ and vice versa. It follows that maximal edge groups of $G$ are isomorphic to maximal edge groups of $B$ and vice versa.  Now factoring out normal closures of all vertex groups in $\widehat G$ and $\widehat B$ we obtain that $\widehat\pi_1(\Gamma)\cong\widehat \pi_1(\Delta)$ (see \cite[Corollary 3.9.3]{R}). But $|E(\Gamma)|=|V(\Gamma)|+rank(\pi_1(\Gamma))-1=|V(\Delta)|+ rank(\pi_1(\Delta))-1=|E(\Delta)|$ (see \cite[Proposition 3.5.3 (b)]{R}). Therefore $|E(\Gamma)|=|E(\Delta)|$
and we can define a bijection $\epsilon:E(\Gamma)\longrightarrow E(\Delta)$ such that for every  maximal edge group $\G(e)$ one has   $\G(e)\cong \B(\epsilon(e))$.

\medskip
Now we shall argue by induction on the maximal order $n$ of an edge group of $(\G,\Gamma)$. From now on we think of $\widehat G=\widehat B$ as the same group and think of $G,B$ as dense subgrous of it.

 If all edge groups are trivial, then the result follows from two preceding paragraphs that gives the base of induction.  Let $\Omega$  be the  subgraph of $\Gamma$ such that $\G(e)$  have order $n$ for each $e\in \Omega$.  For a connected component $C$ of $\Omega$ we denote by $(\G,C)$  the subgraph of groups of $(\G,\Gamma)$ restricted to $C$. We collapse each connected component $C$ of $\Omega$ to a vertex and put on the obtained vertex $v_C$ the fundamental group $\pi_1(\G,C)$   leaving the rest of edge and vertex groups unchanged. Let $(\overline\G,\overline\Gamma)$ be the obtained graph of groups.  Then $G=\pi_1(\overline\G, \overline \Gamma)$   and similar we obtain splitting of $B$ as the fundamental group $\pi_1(\overline \B,\overline\Delta)$.  The edge groups of these graphs of groups have order less than $n$. Let $\overline S(\widehat G), \overline S(\widehat B)$ be the standard profinite trees associated with $(\overline\G,\overline\Gamma)$ and $(\overline \B,\overline\Delta)$ respectively.  By Lemma \ref{relative main} the profinite completion $\widehat{\G(v_C)}$  of the vertex group  of each  collapsed connected component fixes a vertex in the standard profinite tree $\overline S(\widehat B)$ and vice versa.   Hence by the induction hypothesis there is a bijection $\bar\epsilon:E(\overline\Gamma)\longrightarrow E(\overline\Delta)$ such that for every edge group $\overline\G(e)$ of $(\overline\G,\overline\Gamma)$ one has   $\overline \G(e)\cong \overline\B(\bar\epsilon(e))$. Now combining $\overline\epsilon$ with $\epsilon$  we obtain the result.
\end{proof}

{\it Proof of Theorem \ref{JSJ-decomposition}}.  We identify $\widehat G$ with $\widehat B$ and view $G$ and $B$ as dense subgroups of $\widehat G$.  Since any vertex group $\B(v)$ is an $\Oe$-group,  $\widehat{\B(v)}$ is conjugate into $\widehat{\G(w)}$ for some $w\in V(\Gamma)$ and since any vertex group $\G(w)$ is an $\Oe$-group,  $\widehat{\G(w)}$ is conjugate into $\widehat{\B(v)}$ for some $v\in V(\Delta)$.  Hence the result follows from Proposition \ref{maximal edge groups}.

\bigskip


{\it Proof of Corollary \ref{free product}}.  Refining the free decomposition if necessary
and collecting free factors isomorphic to $\Z$
we can obtain $G=*_{j=1}^k G_j * F_0$,  where each $G_j$ is indecomposable into a free product group not isomorphic to $\Z$ and $F_0$ is a free group of finite rank.  Similarly we decompose   $B=*_{i=1}^l B_i* F$.  Then $\widehat G=\amalg_{j=1}^k \widehat G_j \amalg \widehat F_0$, and   $\widehat B=\amalg_{i=1}^l \widehat B_i\amalg \widehat F$.  We look at theses decompositions as profinite fundamental groups of graph of groups, with $G_i,B_j$ vertex groups and generators  of $\widehat F_0$ and $\widehat F$ representing loops of the underlying graphs.  Fix an isomorphism $\nu:\widehat G\longrightarrow \widehat B$.  Let $S(\widehat G)$ and $S(\widehat B)$ be the standard profinite trees on which they act respectively.
 Note that ${\nu(G_i)}$ is either an OE-group (and so is a $\widehat{OE}$-group) or splits  as a free product with amalgamation or as HNN-extension over a non-trivial finite group and so  by Lemma \ref{relative main} its profinite completion $\nu(\widehat G_i)=\widehat{\nu(G_i)}$  fixes a vertex in the standard profinite tree $S(\widehat B)$. Hence  ${\nu(\widehat G_i)}$  is conjugate into some $\widehat B_j$.  Similarly,  $\nu^{-1}(\widehat B_j)$   fixes a vertex in the standard profinite tree $S(\widehat G)$ and therefore is conjugate into some $\widehat G_k$.  Then by Proposition \ref{maximal edge groups}  $l=k$ and $\widehat G_i\cong \widehat B_j$ up to renumeration.  It remaines to observe that $\nu$ induces an isomorphism $$\bar\nu:\widehat G/\langle\langle \widehat G_i\mid i=1,  \ldots, k\rangle\rangle\cong \widehat F_0\longrightarrow \widehat B/\langle\langle \widehat B_j\mid j=1,\ldots, k\rangle\rangle\cong \widehat F.$$  Hence $rank(F)=rank(F_0)$ as needed.  

\bigskip

{\it Proof of Theorem \ref{multiplicative}}. We may assume that $G_i$ are indecomposable free factors of $G$. By Proposition \ref{closed for constructions} $G\in \A$.  Let $B\in \A$ be a group such that $\widehat B\cong \widehat G$. By Corollary \ref{free product} $B=*_{i=1}^n  B_i$ with $\widehat B_i=\widehat G_i$. The number of isomorphism classes of such $B_i$ is exactly $g(G_i,A)$ for each $i=1, \ldots, n$, so $g(G,\A)=g(G_1,\A)\cdots g(G_n,\A)$. 

\bigskip

Since finitely generated torsion free nilpotent groups of class 2 and Hirsch length $\leq 5$ are determined by their profinite completion (see [GS]) we deduce the following 

\begin{cor}\label{nilpotent} If $G_i$ are   finitely generated torsion free nilpotent groups of class 2  and Hirsch length $\leq 5$ then $g(G,\A)=1$.

\end{cor}

\begin{rem} In \cite{BMRS20}, \cite{BMRS18} was proved that certain triangle groups  and certain 3-manifold groups  are profinitely rigid. Hence the free factor of Corollary \ref{nilpotent} can be also these groups with the same conclusion.

\end{rem}

vagnerbessa@ufv.br

ander.porto@ict.ufvjm.edu.br

pz@mat.unb.br
\end{document}